\documentclass[12pt]{amsart}
\usepackage{amsmath}
\usepackage{palatino}

\usepackage{amsfonts}
\usepackage{amsthm}
\usepackage{amssymb}
\usepackage{amscd}
\usepackage[all]{xy}
\usepackage{enumerate}

\theoremstyle{plain}

\newtheorem{thm}{Theorem}[section]

\newtheorem{prop}[thm]{Proposition}

\newtheorem{cor}[thm]{Corollary}
\newtheorem{lem}[thm]{Lemma}

\theoremstyle{definition}
\newtheorem{defn}[thm]{Definition}

\newtheorem{expl}[thm]{Example}

\newtheorem*{ackn}{Acknowledgment}

\newtheorem{rmk}[thm]{Remark}
\newcommand{\sA}{\mathcal{A}}
\newcommand{\sB}{\mathcal{B}}

\newcommand{\sE}{\mathcal{E}}
\newcommand{\sF}{\mathcal{F}}
\newcommand{\sG}{\mathcal{G}}

\newcommand{\sI}{\mathcal{I}}

\newcommand{\sL}{\mathcal{L}}

\newcommand{\sM}{\mathcal{M}}
\newcommand{\sO}{\mathcal{O}}

\newcommand{\sT}{\mathcal{T}}

\newcommand{\sV}{\mathcal{V}}

\newcommand{\mE}{\mathbb{E}}
\newcommand{\mF}{\mathbb{F}}

\newcommand{\mK}{\mathbb{K}}

\newcommand{\mP}{\mathbb{P}}
\newcommand{\mQ}{\mathbb{Q}}
\newcommand{\mT}{\mathbb{T}}

\newcommand{\lp}{\Lambda_\pi}
\newcommand{\corrlp}{ \Lambda_\pi }

\newcommand{\tsc}{Tschirnhausen}
\numberwithin{equation}{section}

\newcommand{\beba} {\begin{equation}\begin{array}{rcl}}

\newcommand{\eaee} {\end{array}\end{equation}}
\title{Generically nef vector bundles on ruled surfaces}

\begin{document}

\maketitle

%\markboth{Beorchia and Zucconi}{slope of rational covers}

\begin{center}
{ Beorchia Valentina$^*$ and Zucconi Francesco$^{**}$}
\end{center}

\bigskip
\begin{small}
\begin{center}{$^*$ Dipartimento di Matematica e Geoscienze,\\
  Universit\`a di Trieste,
via Valerio 12/b, 34127 Trieste, Italy\\
\it{beorchia@units.it}}
\end{center}
\end{small}
% \author{Francesco Zucconi}

\medskip
\begin{small}
\begin{center}{$^{**}$Dipartimento di Scienze Matematiche, Informatiche e Fisiche, \\
Universit\`a degli studi di Udine\\
33100 Udine, Italy\\
\it{francesco.zucconi@uniud.it}}
\end{center}
\end{small}

\begin{abstract} The present paper concerns the invariants of generically nef vector bundles on ruled surfaces.
By Mehta - Ramanathan Restriction Theorem and by Miyaoka characterization of semistable vector bundles on a curve, the generic nefness can be considered as a weak form of semistability.
We establish a Bogomolov type inequality for generically nef vector bundles with nef general fiber restriction on ruled surfaces with no negative section, see Theorem \ref{valeviehweg_nonbal}.
This gives an affermative answer in this case to a problem posed by Th. Peternell
in \cite{P}. 

Concerning ruled surfaces with a negative section, we prove a a similar result for generically nef vector bundles, with nef and balanced general fiber restriction and with a numerical condition on first Chern class, which is satisfied, for instance, if in its class there is a reduced divisor, see Theorem \ref{nuovo_balan}.

Finally, we use such results to bound the invariants of curve fibrations, which factorize through finite covers of ruled surfaces.
\end{abstract}

\tableofcontents

%%%%%%%%%%%%%%
%%%%%%%%%%%%%

\section{Introduction}

The present paper concerns the invariants of generically nef vector bundles on ruled surfaces, see Definition \ref{gennef}.
By Mehta - Ramanathan Restriction Theorem \ref{mr} and by Miyaoka characterization of semistable vector bundles on a curve, Theorem \ref{miy}, the generic nefness can be considered as a weak form of semistability, see Corollary \ref{cormiy}.

A question about a possible relation between generic nefness and 
Bogomolov type inequality has been posed by Th. Peternell in the paper \cite[remarks after Theorem 3.8]{P}. By considering the Harder - Narasimhan filtration of a vector bundle $\sE$ on a projective surface, Miyaoka proved the inequality
$
c_2 (\sE) \ge 0,
$
provided that $\sE$ is generically nef and $c_1(\sE)$ is nef.

In the present paper, under the hypotheses of generic nefness and the nefness of the generic fiber restriction, 
we give an affirmative answer to Peternell's question for 
ruled surfaces with invariant $e =-C_0^2 \le 0$ (
see Theorem \ref{valeviehweg_nonbal}):

\begin{thm}\label{valeviehweg_nonbal} 
Let $Y$ be a ruled surface on a smooth curve $B$ with invariant $e=-C_0^2\le 0$.
Let $\sE$ be a generically nef vector bundle of rank $r$ on $Y$ with nef generic fiber restriction.
 Then
$$
c_2 (\sE) \ge  {\sum _{i=1}^{r-1} a_i \over 2a} c_1(\sE)^2,
$$
where $c_1(\sE)\equiv aC_0+\delta L$ and $(a_1,\dots,a_r)$ is the generic splitting type of $\sE$.

\end{thm}

In the $e>0$ case, we prove a similar bound under two additional assumptions,
namely that
the general fiber restriction is balanced, and $
c_1(\sE)\cdot C_0 \ge -{e\over 2}
$,
see Theorem \ref{nuovo_balan}. The last condition is satisfied, for instance, if 
$c_1(\sE)$ is nef or if $c_1(\sE)$
is effective and $C_0$ is not contained in the base locus of $|2\ c_1(\sE) - C_0|$; this assumption
is typically satisfied by
the \tsc\ sheaf associated with a finite cover of smooth surfaces with reduced branch divisor.
Our results allow us to obtain some bounds on the invariants of 
fibred surfaces factoring through finite covers. In these specific cases the bounds found are better
than the recent bounds found by X. Lu and K. Zuo in \cite{LZ}. Moreover, in the case of primitive cyclic covers $ \pi :S \to Y$,
we obtain the same bound $\lambda_{g,0,n}$ given in \cite [Remark 4.4]{E}, see Theorem \ref{primcyc}.

The techniques involved concern vector bundles and algebraic surfaces techniques. We believe that our approach can be the starting point for further reaserch in the theory of vector bundles on fibred surfaces in general.

\begin{ackn}
This research is supported by national MIUR funds, 
PRIN project {\it Geometria delle variet\`a algebriche} (2015), 
coordinator A. Verra, and by GNSAGA of INdAM.

The first author is also supported by national MIUR funds
FINANZIAMENTO ANNUALE INDIVIDUALE DELLE ATTIVIT\`A BASE DI RICERCA - 2018.

 The second author is supported by Universit\`a degli Studi di Udine - DIMA project Geometry PRIDZUCC2017. 
 
 The examples contained in Subsection \ref{subonesharp} give an answer to a question posed by Ciro Ciliberto at the conference New Trends in Algebraic Geometry, Universita\`a della Calabria, June 12-14 (2013),
about the sharpness of some slope estimates.

 \bigskip
 \noindent 2010 Mathematics Subject Classification 14J60, 14D06

\end{ackn}

\section{Notation and preliminaries}

Let us introduce the definitions involved in the notion of {\it generic nefness}, see \cite [Definition 3.1]{P}).

\begin{defn}\label{gennef}
A vector bundle $\sE$ on a smooth curve is {\it nef} if the tautological divisor of $\mP(\sE)$ is nef.

A vector bundle $\sE$ on a projective variety $X$ of dimension $n\ge 2$ is called 
{\it generically nef} with respect to an ample divisor $H$ if the restriction $\sE_{\vert C}$ is nef
for a general curve $C=D_1 \cap \dots \cap D_{n-1}$, where $D_i\in 
\vert m_i \ H\vert$ are general and $m_i >>0$; such a curve $C$ is said to be {\it MR - general}, which means general in the sense of Mehta-Ramanathan, 
with respect to $H$ (w.r.t. $H$).

A vector bundle $\sE$ is {\it generically nef}
if for every ample divisor $H$
on $X$, the restriction $\sE_{\vert C}$ is nef
for a MR - general curve w.r.t. $H$.

\end{defn}

We now recall Mehta - Ramanathan Restriction Theorem and Miyaoka characterization of semistable vector bundles on a curve, which imply that the generic nefness can be considered as a weak form of semistability.

\begin{thm} (Mehta - Ramanathan Restriction Theorem) 
\label{mr}

A locally free sheaf $\sE$ on a projective variety $X$
 is semi-stable w.r.t. an ample divisor $H$ if and only $\sE_{\vert C}$ is semi-stable for $C$ MR-general w.r.t. $H$.
\end{thm}

\begin{thm} (Miyaoka) \label{miy}
Let $C$ be a smooth curve and $E$ a vector bundle on $C$. Then $E$ is semi-stable if and only if the $\mQ$-bundle $E \otimes {det \ E^\vee\over {\rm rk}\ E}$ is nef, where $E^\vee =Hom(E,\sO_C)$ is the dual vector bundle.
\end{thm}

\begin{cor}\label{cormiy}
If $\sE$ is semi-stable w.r.t. $H$ and if $c_1(\sE) \cdot H \ge 0$, then $\sE$ is generically nef w.r.t. $H$.
\end{cor}

Finally, since we shall consider vector bundles on ruled surfaces, we can talk of the general splitting type.

\begin{defn}
Let $p:Y \to B$ be a ruled surface over a smooth curve $B$,
and let $\sE$ be a rank $r$ vector bundle on $Y$.
We say that $(a_1,\dots,a_r)$, with $a_1 \le \dots \le a_r$
is the {\it generic splitting type} of $\sE$ if for a general fiber $L$ of $p$ we have
$$
\sE \otimes \sO_L \cong \bigoplus_{i=1}^r \sO_{\mP^1} (a_i).
$$
We say that a fiber $L$ is a {\it jumping line} if
$$
\sE \otimes \sO_L \not \cong \bigoplus_{i=1}^r \sO_{\mP^1} (a_i).
$$
Finally, we say that $\sE$ is {\it uniform} if it has no jumping lines.
\end{defn}

Finally, let us recall that 
for a tensor product $\sV \otimes \sL$, where $\sV$ is a rank $r$ vector bundle and $\sL$ is a line bundle,
we have:
\begin{equation}\label{c2tensor}
c_1 (\sV \otimes \sL)= c_1(\sV) +r \ c_1(\sL), \qquad c_2 (\sV \otimes \sL)= c_2 (\sV) +(r-1) c_1 (\sV) \cdot c_1 (\sL) + {r(r-1)\over 2} c_1 (\sL)^2.
\end{equation}
\section{Bogomolov type inequalities }

 \begin{thm}\label{valeviehweg_nonbal} 
Let $Y$ be a ruled surface on a smooth curve $B$ with invariant $e=-C_0^2\le 0$.
Let $\sE$ be a generically nef vector bundle of rank $r$ on $Y$ with nef generic fiber restriction.

 Then
$$
c_2 (\sE) \ge \left( {\sum _{i=1}^{r-1} a_i \over 2a}\right) c_1(\sE)^2,
$$
where $c_1(\sE)\equiv aC_0+\delta L$ and $(a_1,\dots,a_r)$ is the generic splitting type of $\sE$.

\end{thm}

\begin{proof}
We prove the statement by induction on $r\ge 2$. 
If $r=2$, the claimed inequality has been proved in \cite[Theorem 2.8]{BeZ}, in the slightly different setting of blown up ruled surfaces, and under the assumption on $\sE$ to be weakly positive and nef general fiber restriction. 
We observe, however, that the weak positivity of a vector bundle implies its generic nefness. Moreover, since the weak positivity is preserved for any quotient bundle, the same properties hold for such quotients.
Then it is simple to check that the proof of the statement in the rank $2$ case can be done with some little adaptations.

So let us assume $r \ge 3$ and suppose that the claim holds for any vector bundle of rank $q \le r-1$ satisfying the assumptions of the statement.

 We can consider the push - pull map
$$
p^\star p_\star \sE (-a_r\ C_0) \to \sE (-a_r\ C_0);
$$
such a map is generically injective,
hence it is an injective map of locally free sheaves.
Moreover, the quotient sheaf is locally free outside a subscheme 
$Z$ of codimension $2$.

It follows that we have a Brosius type sequence (see \cite{Br}):
\begin{equation}\label{brosiusgeneral}
0\to p^\star p_\star \sE (-a_r \ C_0) \to \sE (-a_r\ C_0) \to \sG \otimes \sI_Z \to 0.
\end{equation}
By setting $\sA:= 
p_\star (\sE (-a_r \ C_0))$, $A:= c_1 (\sA)=$ and $\alpha:={\rm{deg}}(A)$, and by tensoring by $\sO_Y(a_r \ C_0)$ we get:
\begin{equation}\label{brosiusgeneral2}
0\to (p^\star \sA )(a_r \ C_0) \to \sE  \to \sG (a_r \ C_0)\otimes \sI_Z \to 0,
\end{equation}
where $\sG$ is a vector bundle of rank $q\le r$ and $\sA$ is a vector bundle of rank $r-q\ge 1$. We observe that $r-q$ is equal to the number of integers in the general splitting type $(a_1,\dots,a_r)$, which are equal to $a_r$, so that we have
\begin{equation}\label{a}
a= \sum_{i=1}^r  a_i = \sum _{i=1}^q a_i + (r-q)a_r.
\end{equation}

This sequence gives
\begin{equation}\label{formula}
c_2 (\sE ) = c_1(p^\star A\ (a_r \ C_0))\cdot c_1 (\sG(a_r \ C_0)) + c_2(p^\star A\ (a_r \ C_0))+c_2 (\sG(a_r \ C_0)) + Z.
\end{equation}
Set $\sM :=\sG(a_r \ C_0)$. Let us compute each term appearing in (\ref{formula}).
\begin{equation}\label{c1m}
\begin{array}{lll}
c_1(\sM)&\equiv &\Big( c_1(\sE) - p^\star A - (r-q)a_r C_0\Big)\equiv 
a_\sM\ C_0 + (\delta-\alpha)L\\
\end{array}
\end{equation}
where we have set
$$
a_\sM := \sum _{i=1}^{q} a_i,
$$
and
\begin{equation}
\label{prodc1}
\begin{array}{lll}
c_1(p^\star A\ (a_r \ C_0))\cdot c_1 (\sM)&=&
\Big(p^\star A + (r-q)a_r\ C_0\Big) \cdot 
\Big( a_\sM\ C_0 + (\delta-\alpha)L\Big)=\\
&=& (r-q)a_r \  a_\sM \ C_0^2 +(  \alpha \ a_\sM + (r-q)a_r(\delta - \alpha).
\end{array}
\end{equation}

By taking into account the relations (\ref{c2tensor}), we get
\begin{equation}\label{c2a}
\begin{array}{lll}
c_2 (p^\star A\ (a_r \ C_0))&=&(r-q-1)p^\star A \cdot (a_r \ C_0)
+{(r-q)(r-q-1)\over 2}a_r ^2 \ C_0^2=\\
&=& (r-q-1)\alpha\ a_r+{(r-q)(r-q-1)\over 2}a_r ^2 \ C_0^2.\\
\end{array}
\end{equation}

Moreover, since $\sM$ is a quotient of $\sE$ away from the zero - dimensional scheme $Z$, $
\sM$ is a generically nef vector bundle of rank $q < r$. Now we analyze the general fiber restriction of $\sM$. Since $\sE$ has nef general fiber restricyion, we have that the general splitting type $(a_1,\dots,a_r)$ satisfies
\begin{equation}\label{ai}
0\le a_1 \le \dots \le a_r.
\end{equation}
Moreover,
by construction the generic splitting type of $\sM$ is
$$
(a_1,
\dots, a_{q})
$$
where the integers $a_i$ are the first $q$ integers appearing in
(\ref{ai}), hence also $\sM$ is nef on the generic fiber restriction.

It follows that $\sM$
satisfies the assumptions of the Theorem and we can apply the induction hypothesis, which gives the inequality
\begin{equation}\label{diseqM}
c_2(\sM) \ge \left( {\sum _{i=1}^{q-1} a_i \over 
2\ a_\sM}\right) c_1(\sM)^2.
\end{equation}
 By (\ref{c1m}) we get
 \begin{equation}
 c_1(\sM)^2=a_\sM^2 \ C_0^2+
 2\  a_\sM (\delta - \alpha),
 \end{equation}
 and by observing that 
 $$
 \sum _{i=1}^{q-1} a_i  = a_\sM -a_q,
 $$
 the inequality (\ref{diseqM}) becomes
 \begin{equation}\label{diseqMbis}
 c_2(\sM) \ge \left( {a_\sM -a_q \over 
2\ }\right) \left(a_\sM \ C_0^2+
 2\   (\delta - \alpha)\right).
 \end{equation}
 By taking into account (\ref{prodc1}), (\ref{c2a}),
 (\ref{diseqMbis}) and the fact that $Z$ is effective, from (\ref{formula}) we get 
 \begin{equation}\label{quasifinito}
 \begin{array}{lll}
 c_2(\sE) &\ge &(r-q)a_r\ a_\sM \ C_0^2 + \alpha a_\sM+ (r-q)a_r(\delta - \alpha) +\\
 & & + (r-q-1)\ a_r\ \alpha+{(r-q)(r-q-1)\over 2}a_r ^2 \ C_0^2 + 
 \left( {(a_\sM - a_q) \over 
2\ }\right)a_\sM \ C_0^2+
 (a_\sM - a_q)    (\delta - \alpha). \\
 \end{array}
 \end{equation}
 We can rewrite the inequality above in the form
 \begin{equation}\label{sugg_fra}
 c_2(\sE)\geq d C_0^2 +(a_q-a_r)\alpha + c\delta,
 \end{equation}
where
$$
d:=(r-q)a_r\ a_\sM+{1\over 2}(r-q)(r-q-1)a_r ^2+ {1 \over 
2\ }(a_\sM - a_q)\ a_\sM,
$$
$$
c:=(r-q)a_r+a_\sM - a_q.
$$
In particular, the coefficient of $\alpha$ in the expression (\ref{sugg_fra}) is
 $$
 a_q - a_r <0.
 $$
Now we shall bound the integer $\alpha$. Recall that a generically nef vector bundle has a generically nef first Chern class, since exterior products of nef vector bundles are nef (see, for instance, \cite[Theorem 6.2.12, (iv)]{La2}). 
 Hence $c_1(\sM)$ is generically nef.
 So for any nef divisor $xC_0+yL$ with $x\ge 1$ and $y 
 \ge -{1\over 2} x\ C_0^2$, we have
 $$
 c_1(\sM)\cdot (xC_0+yL) \ge 0,
 $$
 that is
 $$
 \left(a_\sM\ C_0 + (\delta-\alpha)L\right)\cdot(xC_0+yL)=a_\sM(xC_0^2+y)+x(\delta-\alpha) \ge 0,
 $$
 which gives
 $$
 \alpha \le \delta +{a_\sM(xC_0^2+y)\over x},
 $$
 and since this holds for any $x\ge 1$ and any $y\ge -{1\over 2} x\ C_0^2$,
 we get
 $$
\alpha   \le \delta + {1\over 2} a_\sM\ C_0^2.
 $$
 By substituting the righthand expression in (\ref{sugg_fra}), we get
 \begin{equation}\label{forza}
 c_2(\sE) \ge (d+ {1\over 2} (a_q-a_r)\ a_\sM)\ C_0^2 +(c+a_q-a_r)\ \delta. 
 \end{equation}
 
 Next we shall suitably express the integer $\delta$. As $c_1 (\sE)^2 =a^2 C_0^2 + 2a \delta$, we can write
$$
\delta = -{a \over 2} C_0^2 + {c_1(\sE)^2 \over 2a}.
$$
Then the inequality (\ref{forza}) becomes
\begin{equation}\label{ecco}
c_2(\sE) \ge (d+{1\over 2} (a_q-a_r)\ a_\sM-{a \over 2}(c+a_q-a_r))C_0^2 +{(c+a_q-a_r)\over 2a}c_1(\sE)^2.
 \end{equation}
 By computing the coefficients in (\ref{ecco}) we get that the coefficient of $C_0^2$ is zero an we get
 $$
c_2(\sE) \ge \left({a_\sM +(r-q-1)a_r\over 2a}   \right) c_1(\sE)^2 =      
\left({a-a_r\over 2a}   \right) c_1(\sE)^2,
$$
 which is the bound in the statement.

 \end{proof}

As a particular case, we can consider generically nef vector bundles 
$\sE$ with nef 
and {\it balanced} general fiber restriction, that is the restriction of $\sE$ to a general fiber of $Y$ is a balanced vector bundle with splitting type
$
(m,\dots,m,m+1,\dots,m+1)$.

\begin{cor}\label{cor}

Let $Y$ be a ruled surface on a smooth curve $B$ with invariant $e=-C_0^2\le 0$.
Let $\sE$ be a generically nef vector bundle of rank $r$ on $Y$, such that the restriction of $\sE$ to a general fiber of $Y$ is a nef and balanced vector bundle with splitting type
$
(m,\dots,m,m+1,\dots,m+1)$, and set $c_1(\sE)\equiv aC_0+\delta L$, $a =mr+k$ and $1\le k \le r-1$.

Then
$$
c_2 (\sE) \ge \left({a-(m+1)\over 2a}\right)c_1(\sE)^2=
\left({r-1\over 2r}-{r-k\over 2ar}\right)c_1(\sE)^2.
$$

Moreover, the equality holds if and only if $\sE$ is uniform and $\sE$
is an extension sitting in an exact sequence of the type
$$
0\to p^\star \sB \otimes \sO_Y((m+1)C_0)\to
\sE \to p^\star \sV\otimes \sO_Y(mC_0)\to 0,
$$
where $\sB$ is a rank $k$ vector bundle on $B$ satisfying
$
{\rm deg}\ c_1(\sB)= \delta,
$
and $\sV$ is a rank $r-k$ vector bundle on $B$ satisfying
$
c_1(\sV)\equiv 0.
$
\end{cor}

\begin{proof}
The bound is a direct consequence of the general bound. The caracterization of the vector bundles attaining the equality can be directly obtained by imposing the equalities in all the bounds considered in the proof of Theorem \ref {valeviehweg_nonbal}.

\end{proof}

We remark that vector bundles with balanced general fiber restriction are the natural generalization of vector bundles with semistable general fiber restriction, in which case the Bogomolov discriminant is non negative by Moriwaki's Theorem
 \cite[Theorem 2.2.1]{Mo}, which we recall.

\begin{thm} \label{stoppino} Let $\varphi\colon Z \to C$ be a fibration from a smooth surface $Z$ to a smooth curve $C$. 
 Let $\sE$ be a torsion free sheaf on
$Z$ such that the restriction of $\sE$ to a general fibre $F\subset Z$ is a $\mu$ - semistable locally free sheaf. 
Then the Bogomolov discriminant $\Delta (\sE)$ satisfies
$$
\Delta (\sE) = c_2 (\sE) - {rk (\sE) -1\over 2\ rk (\sE)} c_1 (\sE) ^2 \ge 0.
$$
\end{thm}

Since in Moriwaki Theorem the only assumption is the semistability of the general fiber restriction, one
can wonder, if the generic balancedness condition could be sufficient in order to have a Bogomolov type inequality.
A negative answer is given the following example.

\begin{expl}\label{esempio_nonbasta}
On $Y=\mP^1 \times \mP^1$ consider the rank two split vector bundle
$$
\sE = \sO_{Y} (m,b_1)\oplus \sO_{Y} (m+1,b_2).
$$
We have
$$
c_1(\sE)^2 =2(2m+1)(b_1+b_2), \qquad c_2(\sE)=m(b_1+b_2)+b_1,
$$
so we see that $c_2(\sE)$ can be arbitrarily lowered by adjusting $b_1$, even with $c_1(\sE)^2$ fixed.
\end{expl}

Let us conclude this section with a result in the $e>0$ case. We will consider 
only the balanced case, and we will need to assume $c_1(\sE)\cdot C_0 \ge -{e\over 2}$, which is satisfied, for instance, if $c_1(\sE)$ is nef, or if $c_1(\sE)$
is effective and $C_0$ is not contained in the base locus of $|2\ c_1(\sE) - C_0|$. 
The last condition is typically satisfied by \tsc \ sheaves associated with surface covers with reduced branch divisor.

\begin{thm}\label{nuovo_balan}
Let $Y$ be a ruled surface on a smooth curve $B$ with invariant $e=-C_0^2> 0$.
Let $\sE$ be a generically nef vector bundle of rank $r$ on $Y$, such that the restriction of $\sE$ to a general fiber of $Y$ is a nef and balanced vector bundle with splitting type
$
(m,\dots,m,m+1,\dots,m+1)$, $c_1(\sE)\equiv aC_0+\delta L$, $a =mr+k$ and $1\le k \le r-1$.

Assume, moreover, that
$
c_1(\sE)\cdot C_0 \ge -{e\over 2}
$.

Then

\begin{equation}
c_2(\sE)\ge \left( {a-(m+1)\over 2a}- {a-k(m+1)\over 2a(a-1)} \right)c_1(\sE)^2.
\end{equation}
\end{thm}

\begin{proof} The proof is similar to the one of Theorem \ref {valeviehweg_nonbal}.
By assumption 
the restriction of $\sE$ to a general fiber $L$ of $p\colon Y\to \mP^1$ 
 is balanced. Since by hypothesis $c_{1}(\sE)\cdot L=a=mr+k$, 
the general fiber restriction of $\sE$ is of the type
$$
\sE_{|L} \cong \bigoplus ^k \sO _{\mP ^1} (m+1) \oplus \bigoplus ^{r-k} 
\sO _{\mP ^1} (m).
$$
Then we have again a Brosius type exact sequence:
\begin{equation}\label{brosius_negativo}
0\to p^\star p_\star \sE (-(m+1)C_0) \to \sE (-(m+1)C_0) \to \sG \otimes \sI_Z \to 0,
\end{equation}
where $p^\star p_\star \sE (-(m+1)C_0) $ has rank $k$ and $\sG$ has rank $(r-k)$.

Set $\sA:= 
p_\star \sE (-(m+1)C_0)$, $A:= c_1 (\sA)=$, $\alpha:={\rm{deg}}(A)$ and $\sM:=\sG((m+1)C_0)$,
so that (\ref{brosius_negativo}) becomes
\begin{equation}
0\to p^\star \sA ((m+1)C_0) \to \sE  \to \sM \otimes \sI_Z \to 0,
\end{equation}
and
$$
c_1(\sM)= c_1(\sE) - p^\star A -k(m+1)C_0\equiv
(r-k)m\ C_0+(\delta -\alpha)L.
$$
The main difference in the present proof is the bound on $c_2(\sM)$.
Since the restriction of the Brosius exact sequence (\ref{brosius_negativo}) to the 
general fiber $L\cong \mP^1$ gives
$$
0 \to \bigoplus ^k \sO _{\mP^1} \to \bigoplus ^k \sO _{\mP ^1} \oplus \bigoplus ^{r-k} 
\sO _{\mP ^1} (-1) \to \bigoplus ^{r-k} 
\sO _{\mP ^1} (-1) \to 0,
$$
the restriction of $\sG$ to the general fiber of $Y$ is 
$\mu$ - semistable. Since $\sM$ is a twist of $\sG$ the same holds for the general fiber of $\sM$.
Hence $\sM$ is Bogomolov semistable by Moriwaki Theorem \ref{stoppino}, and we have
$$
c_2 (\sM) \ge {(r-k-1)\over 2 (r-k)} c_1 (\sM)^2={(r-k)(r-k-1)\over 2}m^2 \ C_0^2 + (r-k)m(\delta - \alpha).
$$
By observing that with the notations of the proof of
Theorem \ref {valeviehweg_nonbal} we have
$$
a_r = m+1, \qquad q=r-k, \qquad a_\sM = (r-k)m,
$$ 
the relation 
(\ref{quasifinito}) becomes
\begin{equation}\label{semplificata}
c_2(\sE) \ge k(r-k)m(m+1)C_0^2+(\delta -\alpha)k(m+1)+\alpha(r-k)m+(k-1)(m+1)\alpha +
$$
$$
+{k(k-1)\over 2}(m+1)^2 C_0^2+
{(r-k)(r-k-1)\over 2}m^2 \ C_0^2 + (r-k)m(\delta - \alpha),
\end{equation}
which simplifies as
$$
c_2(\sE) \ge \left(k(r-k)m(m+1) +{k(k-1)\over 2}(m+1)^2 +
{(r-k)(r-k-1)\over 2}m^2 \right) C_0^2 -
$$
$$
\qquad - \alpha+(k+(r-1)m)\delta.
$$
Next we use the generic nefness of $c_1(\sE)$ to bound $\alpha={\rm{deg}}(A)$. 
Let $H$ be a very ample divisor of $Y$ which avoids the points of $Z$ arising in the Brosius sequence (\ref{brosius_negativo}). Since $\sE$ is generically nef, and since $\sM\otimes \sI_Z$ is a quotient of $\sE$,
 $\sG ((m+1)C_0)\otimes \sO _{m_0H}$ is nef for $m_0 >>0$, 
hence $c_1(\sM)\otimes \sO _{m_0H})\ge 0$.

An ample divisor on a ruled surface admitting a negative section $C_0$ is of the type $H\in |xC_{0}+yL|$, with $x>0$ and $y> x\ e$. The condition
$c_1 (\sM) \cdot m_0H\ge 0$ gives
\begin{equation} 
\alpha \le \delta +m(r-k){y-xe\over x};
\end{equation}
in particular $\alpha \le \delta +m(r-k){1\over x}$ for any $x>0$, so 
\begin{equation}\label{muscoli}
\alpha \le \delta.
\end{equation}
Moreover, using againg the trick
\begin{equation}\label{delta}
 \delta =-{a\over 2}C_0^2+ {c_1(\sE)^2\over 2a},
\end{equation}
we get
$$
c_2(\sE) \ge \left(k(r-k)m(m+1) +{k(k-1)\over 2}(m+1)^2 +
{(r-k)(r-k-1)\over 2}m^2 -{a\over 2}(k+(r-1)m-1)\right) C_0^2 +
$$
$$
\qquad +{(k+(r-1)m-1) \over 2a}c_1(\sE)^2,
$$
that is
\begin{equation}\label{dai}
c_2(\sE)\ge {(r-k)m\over 2} C_0^2+{a-(m+1) \over 2a}c_1(\sE)^2.
\end{equation}
The last bound is not satisfactory, since $C_0^2<0$,
so 
we finally use the assumption that 
$
c_1(\sE)\cdot C_0 \ge {C_0^2\over 2},
$
which gives
$
\delta \ge ({1\over 2}-a)\ C_0^2.
$
The expression (\ref{delta}) yields
$
C_0^2 \ge -{c_1(\sE)^2 \over a(a-1)},
$
and by (\ref{dai}) we get
$$
c_2(\sE)\ge \left( {a-(m+1)\over 2a}- {a-k(m+1)\over 2a(a-1)} \right)c_1(\sE)^2.
$$

\end{proof}

%%%%%%%%%%%%%%%%%%%%%%%%%%%%%%%%%%%%%%%%%%%
%%%%%%%%%%%%%%%%%%%%%%%%%%%%%%%%%%%%%%%%%%%%%
%%%%%%%%%%%%%%%%%%%%%%%%%%%%%%%%%%%%%%%%%
\section{The normalized relative canonical divisor}

In this section we shall apply the Bogomolov type inequalities to the Tschirnhausen sheaf of a finite cover of a Hirzebruch surface.
Indeed, by the Viehweg Weak Positivity Theorem \cite{V}, the Tschirnhausen sheaf is weakly positive away from the branch locus, and hence nef on the complement of the branch locus (see also \cite{La}), so it is 
in particular generically nef.
This will allow us to bound the relative Euler characteristic $\chi_f$ 
of a fibration factoring through a finite cover.

Moreover, we shall
introduce the {\it normalized relative canonical divisor} of a finite morphism $\pi$ and we shall show that its selfintersection is related with the slope.

We first recall how to determine the invariants and the slope of a fibration factoring through a finite cover.

\begin{defn} Let $\pi:S \to Y$ be a finite cover of degree $n$ between smooth surfaces. Then the sheaf
$
\pi_\star \omega_{S/Y}
$
is locally free of rank $n$ and we can consider the {\it trace} map, which is surjective:
\begin{equation}\label{traccio}
tr :\pi_\star \omega_{S/Y}
 \to \sO _Y.
\end{equation}
 The kernel $\sE$ is a locally free sheaf of rank $n-1$ on $Y$
the exact sequence
\begin{equation}\label{traccia}
0\to \sE \to \pi_\star \omega_{S/Y} \to \sO _Y \to 0
\end{equation}
splits.

Following \cite{Mi}, it is customary to call $\sE$ the {\it Tschirnhausen sheaf} of the finite morphism $\pi$; in fact, Miranda calls Tschirnhausen module the sheaf $\sE^\vee$. 	
	\end{defn}

\begin{lem}\label{scriviamolo}
Let $S,Y$ be smooth surfaces, and let $\pi\colon S \to Y$ be a 
finite cover of degree $n$ with relative canonical divisor 
$K_{S/Y}$. Then in the rational Chow ring $A(Y) \otimes \mathbb Q$ we have
\begin{enumerate}
\item \label{R} $\pi_\star K_{S/Y}\equiv 2 c_1 (\sE)$,
\item \label{chi} $ \chi (\sO_{S})= n \chi (\sO _Y) +{1\over 2} c_1 
(\sE) \cdot K_Y +{1\over 2} c_1 (\sE)^2 - c_2 (\sE)$;
\end{enumerate}
\end{lem}
%%%%%%%%%%%%%%%%%%%%%%%%%%%%%%%%%%%%%%%%%%%%
\begin{proof} 
The first relation is well known in the case of flat finite morphisms.
The result follows from
Grothendieck-Riemann-Roch Theorem applied
to the morphism $\pi\colon S\to Y$ and the sheaf
$\omega_{S/Y}$.

The Grothendieck-Riemann-Roch Theorem asserts that for a proper morphism $\pi$ of smooth varieties we have
$$
{\rm ch} (\pi_! \ \omega_{S/Y}) \cdot {\rm td}\ {\sT}_Y = \pi_\star 
({\rm ch }\ \omega_{S/Y} \cdot {\rm td} \ {\sT}_S).
$$

As $R^1 \pi_\star \omega_{S/Y}=0$ since $\pi$ is finite, we have
$$
\pi_! \ \omega_{S/Y}=\pi_\star \omega_{S/Y}=\sE.
$$
This yields
$$
(n+c_1 (\sE)+{1\over 2}(c_1 ^2 (\sE) -c_2 (\sE)) \cdot (1-{1\over 2}K_Y +\chi(\sO 
_Y))=\pi_\star ((1+ K_{S/Y} + {1\over 2} K_{S/Y}^2)\cdot (1-{1\over 2} K_S+\chi(\sO_S)).
$$

The divisorial part satisfies
$$
c_1 (\sE)- {n\over 2} K_Y=\pi_\star (K_{S/Y}- {1\over 2} K_S).
$$
As $K_S \sim \pi^\star K_Y + K_{S/Y}$, we have $\pi_\star K_S \equiv nK_Y + \pi_\star K_{S/Y}$ and the first claim follows.

The equality between the codimension two cycles gives
formula (\ref{chi}).
\end{proof}

\begin{defn}
A {\it fibration} $f\colon S\to B$ is a flat surjective morphisms between a smooth surface $S$ and a smooth curve $B$ with connected fibers, such that if $x\in B$ is general then $F_x:=f^{-1}(x)$ is a smooth curve.

Following Xiao \cite{X}, we can associate with $f\colon S \to B$ a rational number $s(f)$, called the {\it slope} of $f$, defined as:
$$
s(f): =\frac{K^2_{f}}{{\chi_f}}
$$
\noindent 
where $K_f=K_S -f^\star K_{B}$ is the relative canonical divisor, 
$\chi _f : = {\rm deg}\ f_\star  \omega _f $, and $\omega _f := {\sO}_S(K_f)$.

\end{defn}
\begin{rmk} 
We recall the well known relations:
\begin{equation}\label{k2chi}
K_f ^2=K_S ^2-8(g-1)(g(B)-1), \qquad \chi _f = \chi (\sO _S)-(g-1)(g(B)-1).
\end{equation}
\end{rmk}
\begin{cor}\label{terminivari} 
 Let $f\colon S\to B$ be a
 fibration, which factorises through a finite cover $\pi: S \to Y$ of a ruled surface $Y$.
  Then 
\begin{equation}\label {kfchif}
 K_f ^2
 =K_{S/Y}^2-{4\over (g+n-1)} c_{1}(\sE)^{2}, \qquad
 \chi_f = {(g+n-2)\over 2 (g+n-1)} c_1 ^2(\sE) - c_2 (\sE).
 \end{equation}
 
 \begin{equation}\label{laformula}
 s(f)=\frac{K_f ^2}{ \chi_f}=\frac
 {K_{S/Y}^2-{4\over (g+n-1)} c_{1}(\sE)^{2}}{{(g+n-2)\over 2 (g+n-1)} 
 c_1 ^2(\sE) - c_2 (\sE)}.
 \end{equation}

 \end{cor} 
 
\begin{proof} We can write $K_{S}^{2}=K_{S/Y}^{2} 
 +2K_{S/Y}\cdot\pi^{\star}K_{Y}+nK_{Y}^{2}$, then by projection formula and Lemma \ref{scriviamolo}
 $K_{S}^{2}=K_{S/Y}^2+nc_{1}(\sE)\cdot K_{Y}+nK_{Y}^{2}$.
 As $K_f ^2 = K_S ^2 - 8 (g-1)(b-1)$, where $b =g(B)$, we have
 $$
 K_f ^2 = {K_{S/Y}^2+nc_{1}(\sE)\cdot 
 K_{Y}+nK_{Y}^{2}-8(g-1)(b-1)}.
 $$
Finally, by choosing the generators of the Neron-Severi group of $Y$ to be the 
classes $[C_{0}]$ and $[L]$ where $L\in{\rm{NS}}(Y)$ is the class of a ruling and 
$C_{0}\in{\rm{NS}}(Y)$ is the class of a section of minimal 
selfintersection, we may write
\begin{equation}\label{c1}
c_1(\sE) \equiv (g+n-1)C_0 + \left( {c_1 (\sE)^2 \over 2(g+n-1)} +(g+n-1)\ C_0^2\right)L,
\end{equation} 
and the first formula follows.

Taking into account that $\chi_f= \chi(\sO_S)-(g-1)(b-1)$, 
the formula for $\chi_f$ follows from Lemma \ref{scriviamolo}, (2).
\end{proof}

Now we introduce the normalised relative canonical divisor of a finite cover, and we shall see that it is closely related to the slope
of the induced fibration.
Such a connection is not surprising, as a similar argument has 
already been used in such a context.

For instance,
the Cornalba - Harris theory for bounding the slope of any fibration
$f$
relies 
on the study of the {\it normalized relative canonical} divisor
of a fibration $f\colon S \to B$
$$
\mK _f:=K _{f} -{1\over g} f^\star c_1 (f_\star \omega _{f}), \quad \widetilde \omega _f := \sO _S (\mK _f)
$$
and on the {\it normalized
Hodge bundle}
$$
\mE _f := f_\star \widetilde \omega _f.
 $$
Indeed, the pseudo - effectivity of $f_\star (\mK _f ^2)$, proved by Cornalba and Harris in \cite [Theorem 1.1, Proposition 2.9 and Section 4]{CH},
under the assumption that the Hilbert points of the general fibre are semistable,
is a crucial step in their proof of the classical bound on the slope
$$
s(f) \ge 4- {4\over g}.
$$

 A similar task has been used by Fedorchuk and Jensen in \cite {FJ}, 
 who obtained as a straight consequence of the positivity of 
 $c_1 (f_\star \widetilde \omega _{f} ^{\otimes 2})$ the result that if $S\to B$
 is a flat family of Gorenstein curves with the generic 
 fiber a canonically embedded curve whose $2$nd Hilbert point 
 is semistable (e.g. with the generic fiber a general trigonal curve), 
 then the slope satisfies the inequality $s(f) \ge 5 - {6\over g}$.
 
Also in the context of projective vector bundles
$\pi: \mP (\sG) \to Y$ fibered in $\mP ^{r-1}$ over a variety $Y$ a similar divisor is studied, 
namely the so called {\it normalized tautological divisor}
$$
\mT_\sG := T_{\mP (\sG)} - {1\over r}\pi ^\star c_1 (\sG) = -{1\over r} K_{\mP (\sG)/Y},
$$
and the nefness of such a divisor has been investigated by N. Nakayama \cite{Nak}.
More precisely, Nakayama proved the following result:

\begin{thm}\label{nakayama}
Let $\mathcal{G}$ be a rank $r$ vector bundle on a smooth complex projective
variety $Y$ of dimension $d\ge 2$. Then the following conditions are equivalent:
\begin{itemize}
\item $\mT_\sG$ is nef;
\item $\mathcal{G}$ is $\mu$-semistable and 
$\Big(c_2(\mathcal{G})-{(r-1)\over 2r} c_1(\mathcal{G})^2\Big)\cdot A^{d-2}=0$
for an ample divisor $A$.
\end{itemize}
\end{thm}

It turns out that in our context, as we are dealing with a finite morphism
and a rank $(n-1)$ torsion free sheaf $\sE$ with $c_1 (\sE)=c_1 (\pi_\star \omega _{S/Y})$,
it is natural to give the following Definition:

\begin{defn}\label{normalised relative} Let $\pi\colon S \to Y$ be a finite morphism of degree $n$. The $\mQ$ - divisor
$$
\Lambda _\pi := K_{S/Y} - {1\over (n-1)}\pi ^ \star (c_1 (\pi _\star \omega _{S/Y})).
$$
\noindent is called the {\it normalised relative canonical divisor} of $\pi$.
\end{defn}
 
 To explain the reason which leads to the definition of $\lp$, we need to recall
the following well-known result (see \cite{CE}).

\begin{thm}\label{Casnati}
Let $Y$ be an integral surface and let $ \pi:S \to Y$ be a Gorenstein cover of degree $n\ge 3$. 
There exists a unique $\mP ^{n-2}$-bundle
$\pi _Y:\mP \to Y$ and an embedding $i:S \to \mP$ such that $\pi=\pi _Y \circ i$. Moreover
 $\mP\cong \mP (\sE)$ and the ramification divisor $R$ satisfies:
$$
\sO _S (R) \cong i^\star \sO _{\mP (\sE)} (1).
$$
\end{thm}

 \begin{rmk}
 Consider the embedding $S\subseteq \mP(\sE)$. Since 
 $(T_{\mP(\sE)})_{\vert S}=K_{S/Y}$, we get that
 $$
 \lp = (\mT_\sE )_{S}.
 $$
 \end{rmk}

The connection between the slope and $\lp$ is given by the following:

\begin{prop}\label{facilfacil} 
Let $f:S \to B$ be a fibration, with general fiber $F$ a smooth curve of genus $g$.

Assume that $f$ factorises through a finite degree $n$ cover $\pi:S \to Y$, where $Y$ is a ruled surface,
and assume that the general fiber restriction of the Tschirnhausen sheaf $\sE$ is a twist of the trivial sheaf. 

By setting
\begin{equation}\label{chimax}
\chi_0^{max}:={g\over 2 (n-1)(g+n-1)}c_{1}(\sE)^2,
\end{equation}
we have
 $$
 \chi_0^{max}\ge \chi_f
 $$ 
 and the following equality holds:
\begin{equation}\label{kf}
K_f ^2= \sF (n,g) \chi _0^{max} +\Lambda _\pi ^2,
\end{equation}
where
$$
\sF (n,g)=6-\frac{2}{n-1}-\frac{2n}{g}
$$
is the conjectural bound of Stankova (see \cite[Conjecture 13.3]{S}).

In particular 
\begin{equation}\label{boundslope}
s(f)\geq \sF (n,g) +\frac{\Lambda _\pi ^2}{\chi_0 ^{max}}.
\end{equation}

\end{prop}
 \begin{proof} By Theorem \ref{stoppino} we have
 $c_2(\sE) \ge \frac{n-2}{2(n-1)}c_{1} (\sE)^{2}$, and since by 
 Corollary \ref{terminivari} it holds
 $\chi_f = {(g+n-2)\over 2 (g+n-1)} c_1 ^2(\sE) - c_2 (\sE)$,
 it follows that $\chi_f\leq{(g+n-2)\over 2 (g+n-1)} c_1 
 ^2(\sE)-\frac{n-2}{2(n-1)}c_{1} (\sE)^{2}=\chi_0^{max}$.

 With a direct computation one can obtain the equality (\ref{kf}), and the inequality (\ref {boundslope}) follows
 immediately.
 \end{proof} 
As a consequence, the problem of bounding the slope can be rephrased in a problem of bounding $\lp^2$. 

\bigskip

The Bogomolov - type inequalities given in 
 Corollary \ref{cor} and Theorem \ref{nuovo_balan} allow us to derive a similar result also in the non divisible case $(n-1) \not | g$. 
 Indeed, the Tschirnhausen sheaf is nef outside the branch locus by 
 the Weak Positivity Theorem of Viehweg \cite[3.4]{V}. If we assume that the general fiber restriction of $\sE$ is balanced, which can be rephrased  with the assumption that
 the general fiber of the fibration corresponds to a point outside the Maroni locus of a suitable Hurwitz scheme or of the moduli space, and that the branch divisor is reduced, which corresponds to impose the open condition that $\pi$ has generically simple ramifications, then the Corollary \ref{cor} and Theorem \ref{nuovo_balan} apply.
 
 Therefore, recalling that $\chi_f= {(g+n-2)\over 2 
 (g+n-1)}c_1 ^2(\sE) - c_2 (\sE)$ and setting $g+n-1=(n-1)m+k$, we have 
  $$
  \chi_f \le
 \left\{
 \begin{array}{ll}
\left( {(g+n-2)\over 2 
 (g+n-1)}- {(n-2)\over 2(n-1)}+ {n-1-k\over 
2(n-1)(g+n-1)}\right) c_1 (\sE)^2 &{\rm if}\ C_0^2\ge 0
\\ 
\left( {(g+n-2)\over 2 
 (g+n-1)}-{(g+n-1)-(m+1)\over 2(g+n-1)} + {(g+n-1)-k(m+1)\over 2(g+n-1)(g+n-2)} \right)c_1 ^2(\sE)
&{\rm if}\ C_0^2< 0.\\
\end{array}
\right.
$$
\noindent
that is 
\begin{equation}\label{nuovochimax}
\chi_f \leq \chi_k^{max}:=
\left\{
 \begin{array}{ll}
\frac{m}{2(g+n-1)}c_{1}(\sE)^{2} &{\rm if}\ C_0^2\ge 0
\\ 
\left( \frac{m}{2(g+n-1)} + {(n-1-k)m\over 2(g+n-1)(g+n-2)} \right)c_1 ^2(\sE)
&{\rm if}\ C_0^2< 0.\\
\end{array}
\right.
\end{equation}

Then we can write
$$
K_f^2=K^{2}_{S/Y}-\frac{4}{g+n-1}c_{1}(\sE)^{2}=\Lambda_{\pi}^{2}+\sF(n,g,k)\chi_k^{max} 
$$\noindent
where
$$
\sF(n,g,k):=\left\{
 \begin{array}{ll}
 {6g-2(n-1)\over (g+n-1-k)}-{2\over n-1} -{2k\over (n-1)(g+n-1-k) }&{\rm if}\ C_0^2\ge 0
\\ 
 \frac{(6g-2(n-1))(g+n-2)}{(g+n-1-k)(g+2n-3-k)} - {2(g+n-1)(g+n-2)\over (n-1)(g+n-1-k)(g+2n-3-k)} 
&{\rm if}\ C_0^2< 0.\\
\end{array}
\right.
$$

This shows that the function $\sF(n,g,k)$ can be used to replace the function $\sF(n,g)$ in the non divisible case; more precisely, we have:

\begin{prop}\label{kapnonzer} Let $f\colon S\to B$ a semistable 
 fibration over a rational curve $B$. Assume that $f$ factorizes through a finite cover of degree $n$ of ruled surface and that the Tschirnhausen sheaf is balanced on the general fiber. If the genus of the general fiber of $f$ is $g=(m-1)(n-1)+k$, where 
 $1\leq k\leq n-2$,
 then
 \begin{equation}\label{boundslopek}
 s(f)\ge \sF(n,g,k)+\frac{\Lambda^{2}_{\pi}}{\chi_k^{max}}.
\end{equation}

 \end{prop}

 \begin{rmk}
 The bounds (\ref{boundslope}) and  (\ref{boundslopek}) given in Propositions \ref {facilfacil} and \ref{kapnonzer} hold also
 for fibrations $f:S \to B$, which are the relatively minimal model of fibrations satisfying the given hypotheses. Indeed, it is enough to observe that the selfintersection of the relative 
canonical divisor of a relatively minimal model of a given fibration can not decrease, and the relative Euler characteristic is unchanged.
 \end{rmk}

 %%%%%%%%%%%%%%%%%%%%%%%%%%%%%%%%%%%%%%
 %%%%%%%%%%%%%%%%%%%%%%%%%%%%%%%%%%%%%%%
 %%%%%%%%%%%%%%%%%%%%%%%%%%%%%%%%%%%%%%%

\section{Positivity results on $\Lambda_\pi$}
From the results of the previous section it follows, that 
a bound on $\lp^2$ gives a bound also on the slope.
Therefore we are going to establish some conditions, under which the normalized relative canonical divisor has non-negative selfintersection. 
Some very similar problems have been studied in \cite{BS1} and \cite {BS2},
but their results don't apply to our case.

We first analyse the restriction of the divisor $\lp$ to the general fiber.
\begin{prop}\label{lambdafibra}
Let $F$ be a general fiber of the fibration $f:S \to B$. If $
g=(n-1)(m-1) 
$ and $F$ is not contained in the Maroni locus,
 the restriction of $\lp$ to $F$ satisfies:
\begin{enumerate}
\item
$(\lp) _{|F}\sim
K_F - (m-2)\Gamma_F,
$
where $\Gamma_F\in g^1_n$ is a gonal divisor;

\item $h^0(\sO_F( K_F - (m-2)\Gamma_F))=n-1$;

\item the linear system $| K_F - (m-1)\Gamma_F|$ is base point free.

\end{enumerate}
\end{prop}

\begin{proof}

We have
$$
(\Lambda_\pi) _{|F}= (K_{S/Y} -\pi^\star (mC_0 + kL)) _{|F}=
(K_S - \pi^\star K_Y -\pi^\star (mC_0 + kL)) _{|F}\sim
$$
$$
\sim K_F - (m-2)\Gamma_F,
$$
where $\Gamma_F\in g^1_n$ is a gonal divisor, which proves $(1)$.

Let us compute $h^0( \sO_F( (m-2)\Gamma_F))$ using the Geometric Riemann Roch Theorem:
$$
h^0( \sO_F( (m-2)\Gamma_F))= (m-2)n - \dim \langle (m-2)\Gamma_F\rangle,
$$
where $\langle (m-2)\Gamma_F\rangle \subset \mP^{g-1}$ is the linear span on the canonical model of the curve $F$. Now recall that since $F$ is Maroni general, the canonical model of $F$ lies on $W\cong\mP (\oplus ^{n-1} \sO_{\mP^1}(m-2))$ embedded in $\mP^{g-1}$ by the tautological linear system. It follows that
$$
\dim \langle (m-2)\Gamma_F\rangle = (m-2)(n-2) + m-3,
$$
hence by the Geometric Riemann Roch Theorem we have
$$
h^0( \sO_F( (m-2)\Gamma_F))=(m-2)n-(m-2)(n-2) - (m-3)
=m-1,
$$
hence by Riemann Roch
$$
h^0(\sO_F( K_F - (m-2)\Gamma_F))=h^0( \sO_F( (m-2)\Gamma_F)) -(({\rm deg} (m-2)\Gamma_F) -g+1)=
$$
$$
=m-1-(m-2)n+g-1=n-1,
$$
which proves $(2)$.

Finally, assume by contradiction that $P$ is a base point of the linear system $|K_F - (m-2)\Gamma_F|$. Then by the Geometric Riemann Roch Theorem we would have
$$
\dim \langle K_F - (m-2)\Gamma_F\rangle=\dim \langle K_F - (m-2)\Gamma_F - P\rangle +1,
$$
and
$$
\dim \langle (m-2)\Gamma_F\rangle= \dim \langle (m-2)\Gamma_F + P\rangle.
$$
We claim that the last equality can not be satisfied. Indeed, the subspace
$\langle (m-2)\Gamma_F\rangle$ cuts on $W$ exactly $m-2$ fibers; indeed,
since the minimum degree of a unisecant curve on $W$ is $m-2$, the subspace $\langle (m-2)\Gamma_F\rangle$ contains no horizontal component. It follows that the divisor cut out by $\langle (m-2)\Gamma_F\rangle$ on the canonical model of $F$ is exactly $(m-2)\Gamma_F$.

\end{proof}

\begin{cor}\label{quadrato} Let $f:S \to B$ be a fibration in {\it irreducible} curves, with general fiber $F$ a smooth curve of genus $g$. 
Assume that $f$ factorises through a finite degree $n$ cover $\pi:S \to Y$, where $Y$ is a ruled surface, with Tschirnhausen sheaf generically a twist of the trivial sheaf.

If the restriction map
$
H^0 (\sO_S(\lp)) \to H^0(\sO_F(\lp))
$
is surjective, then
 $
 \lp^{2} \ge 0.
 $
\end{cor}

\begin{proof}

We claim that the linear system $| \Lambda_\pi |$ has no horizontal base locus.

Assume by contradiction that $\corrlp$ has a horizontal component $\varkappa$ in its base locus. Then $\varkappa
_{| F}$ is contained in the base locus of $|(\lp) _{|F}|$. But the latter linear system is base point free by Proposition \ref{lambdafibra}, (3).
This proves that $|\lp|$ has no horizontal base locus.

Finally, since all the fibers of $f$ are irreducible, $\lp$ has no vertical base locus.

Summing up, as $\lp$ is effective and $|\lp|$ has 
no positive dimensional base locus, we have
$
\lp^2 \ge 0.
$

\end{proof}

 \subsection{Rational fibrations with uniform Tschirnhausen sheaf}

In the following Proposition we shall prove that in the divisible case, the fibrations over a rational curve, with uniform and generically balanced Tschirnhausen sheaf and with semistable unisecant restriction, satisfy the assumption of Proposition \ref{lambdafibra}. We remark that by a recent result given in \cite{DP}, a sufficiently general curve in the Hurwitz scheme of degree $n$ covers of curves of
given genus $p \ge 0$ has a $\mu$ - semistable \tsc\ sheaf. Therefore the assumption of semistability on unisecant restrictions can be read as
a generality assumption concerning the family of pull - back curves of 
a family of general unisecant curves.

\begin{prop}\label{ssunif}
Let $f:S \to \mP^1$ be a fibration with irreducible fibres, 
with general fiber $F$ a smooth curve of genus $g$ and gonality $n$, where $n \ge 5$, such that 
$(n-1)|g$ and with balanced reduced gonal direct image sheaf.
Assume that $f$ factorises through a finite morphism
$\pi:S \to Y$, where $p:Y \to \mP^1$ is a Hirzebruch surface.

If the restriction of $\sE$ to some unisecant ample divisor is semistable and if $\sE$ is uniform,
then the restriction map $H^0 (\sO_S(\Lambda_\pi)) \to H^0(\sO_F(\Lambda_\pi))$ is an isomorphism. 

In particular, $s(f) \ge \sF (n,g)$.

\end{prop}

\begin{proof} We set 
$
c_1 (\sE) = (g+n-1) C_0 + \delta L,
$
where $C_0$ is a section with $C_0 ^2\le 0$ and $L$ a fiber of ruling on $Y=\mF_e$.

Since $\sE$ is uniform, the restriction to any fiber $L$ of the ruling satisfies
$
\sE_{|L} \cong \bigoplus _{i=1}^{n-1} \sO_{\mP^1} (m),
$
where $m = {g\over n-1} +1$.
Then the injective map of sheaves 
$
p^\star p_\star \sE (-m C_0) \to \sE (-m C_0) 
$
is an isomorphism, so 
$$
\sE (-mC_0) \cong \bigoplus_{i=1}^{n-1} \sO_Y (p^\star A_i).
$$
By the assumption that $\sE$ is semistable with respect to an ample divisor 
$H \sim C_0 + (e+a)L$, we get 
$
{\rm \deg A_i} = k,
$
for any $i$, and 
$
k= {\delta \over (n-1)} \ge 0,
$
so that 
\begin{equation}\label{c1multipla}
c_1 (\sE) \sim (n-1) (m C_0 + k L).
\end{equation}
We finally get
$$
h^0 (\sO_S (\Lambda _\pi))=h^0 (\pi_\star \sO_S (\Lambda _\pi))\cong
h^0(\sO_Y(-(m C_0 + k L)) \oplus \sE (-(m C_0 + k L)))= h^0 (\oplus _{i=1}^{n-1} \sO_Y)=n-1.
$$
On the other hand,
$$
h^0 (\sO_S (\Lambda _\pi- F))=h^0 (\pi_\star \sO_S (\Lambda _\pi-F))\cong
h^0( \sE (-(m C_0 + (k+1) L)))= h^0 (\oplus _{i=1}^{n-1} \sO_Y(-L))=0.
$$
Hence the restriction exact sequence
$$
0\to \sO_S (\Lambda _\pi- F)\to \sO_S (\Lambda _\pi) \to \sO_F (\Lambda _\pi) \to 0
$$
determines the isomorphism of the statement.
\end{proof}

\begin{rmk} We observe that in the case when $\sE$ is a uniform  vector bundle, that is $\Delta (\sE)=0$, and if $\sE$ is $\mu$-semistable with respect to some ample divisor, then the {\it normalized tautological divisor} $T_{\mP(\sE)} - {1\over (n-1)} p^\star c_1 (\sE)$ where $p\colon \mP (\sE) \to Y$ is nef by Nakayama's Theorem
\ref{nakayama}. It follows that $\left( T_{\mP(\sE)} - {1\over (n-1)} p^\star c_1 (\sE)
\right) _{S} =K_{S/Y} - {1\over (n-1)} \pi^\star c_1 (\sE)=\Lambda_\pi$ is also nef. Since the restriction map is surjective by Proposition \ref{ssunif},
all this implies directly that in such a case $\Lambda _\pi ^2 \ge 0$.

\end {rmk}

\subsection{Upper bounds on $\lp^2$ and primitive cyclic covers}

We recall that for finite covers, we have the following upper bound on $R^2$ in terms of $c_1(\sE)^2$, which is a consequence of the Hodge Index Theorem applied to the $\mQ$ - divisor $R-\pi^\star {2\over n} c_1(\sE)$:

\begin{lem}\label{index}
 Let $\pi\colon S \to Y$ be a Gorenstein cover degree $n$, and
 let $\sE$ be the \tsc\ sheaf. 
Then
\begin{equation}\label{indexbound}
K_{S/Y}^2 \le {4\over n} c_1 (\sE)^2.
\end{equation}
\end{lem}

\begin{proof}
See \cite[Lemma 3.12]{BeZ}.
\end{proof}

\begin{cor}
With the assumptions of Lemma \ref{index} we have
$$
\lp^2 \le {(n-2)^2\over n (n-1)^2}c_1(\sE)^2.
$$

\end{cor}
\begin{proof} Since $\lp = K_{S/Y} - {1\over n-1}\ \pi^\star c_1(\sE)$, we have
$$
\lp^2 = K_{S/Y}^2-2K_{S/Y}\cdot {1\over n-1}\ \pi^\star c_1(\sE)+{1\over (n-1)^2}\ (\pi^\star c_1(\sE))^2,
$$
and by projection formula and by Lemma \ref{scriviamolo}, (1), it follows
$$
\lp^2=K_{S/Y}^2-{4\over n-1} c_1(\sE)^2 +{n\over (n-1)^2} c_1(\sE)^2=
K_{S/Y}^2 - {3n-4\over (n-1)^2} c_1(\sE)^2\le {(n-2)^2\over n (n-1)^2}
c_1(\sE)^2.
$$

\end{proof}
We remark that the Hodge Index Theorem also implies, that the equality holds in (\ref{indexbound}) if and only if
$$
K_{S/Y} \sim {1\over n}\pi^\star 2c_1{\sE}\equiv {1\over n}\pi^\star B_\pi.
$$

Such a condition is satisfied, for instance, when all ramification points of $\pi$ are total ramification points, that is of maximal ramification index $n$.

A typical context, when this happens, is the one of {\it primitive cyclic covers} $ \pi :S \to Y$, that is covers such that there exist an effective divisor $A \subset Y$ and an effective divisor $D\subset S$ such that 
$$
S \cong {\bf Spec} \oplus _{i=0} ^{n-1}\sO_Y (iA),
$$
and such that $\pi:S\to Y$ does not factorise through two covers of smaller degree.
In this case the following hold:
\begin{enumerate}
\item $B_\pi \sim nA$ and $\pi^\star B_\pi = n D$;
\item $K_{S/Y} \sim (n-1)D$;
\item $\pi_\star \sO_S(K_{S/Y})\cong \bigoplus_{i=0}^{n-1} \sO_Y(iA)$;
\item $c_1(\sE)={n(n-1)\over 2} A$, $c_2(\sE)={n(n-1)(n-2)(3n-1)\over 12} A^2$.
\end{enumerate}

From this we obtain that
$$
K_{S/Y} \sim \pi^\star {2\over n} c_1 (\sE), \qquad K_{S/Y}^2 = {4\over n} c_1 (\sE)^2,
\qquad 
\lp^2={(n-2)^2\over n (n-1)^2}
c_1(\sE)^2.
$$

Since $c_1(\sE)={n(n-1)\over 2} A$, if $A^2 \ge 0$,
we get $\lp^2 \ge 0$. 
This gives a bound, which is exactly the bound $\lambda_{g,0,n}$ given in \cite [Remark 4.4]{E}:

\begin{thm}\label{primcyc}
Let $f:S \to B$ be the relatively minimal model of a finite cyclic cover $\pi: \widetilde S \to Y$ of a ruled surface $Y$.

Then
$$
s(f) \ge {24(g-1)(n-1)\over (n^2+4ng-3n+2-2g)}=6- {6\over 2n-1} -{12n(n^2-1)\over 2g(2n-1)+(n-1)(n-2)}.
$$
\end{thm}
\begin{proof}
Recall that the selfintersection of the relative 
canonical divisor of a relatively minimal model of a given fibration can not decrease, and the relative Euler characteristic is unchanged. Then we can apply formula (\ref{kf}).
\end{proof}
\begin{rmk}
We observe that for cyclic covers
the Tschirnhausen sheaf is uniform.
\end{rmk}

\subsection{A Beniamino Segre's construction}

Following closely \cite[Chapter 21 section 12]{ACG} we shall show that for a general 
$[C]\in{\overline{{\mathcal M}^{1}_{g,n}}}$, where
${\overline{{\mathcal M}^{1}_{g,n}}}$ is the closure of the $n$-gonal locus in the
moduli space $\overline{\mathcal M}_g$ of curves of genus $g$,
 if $p\colon C\to\mP^{1}$ is the gonal covering,
then the corresponding \tsc \ sheaf $\sE_{C}$ is balanced. 

\begin{thm}\label{segre} Let $g\geq 3$. For any integer $n$ such that $3\leq n\leq 
 \frac{g}{2}+1$, there exists a smooth curve $C$ of genus $g$ 
 admitting a complete $g^{1}_{n}$ without base points, and 
 such that the corresponding \tsc \ sheaf $\sE_{C}$ is balanced.
\end{thm}
\begin{proof} Let $\mathbb F_{e}$ be a Hirzebruch surface with invariant $e =-C_0^2 \ge 0$. Consider the complete 
 linear system
 $$
 \Sigma_{n,h}=|nC_{0}+hL|.
 $$
 Assume that $h>\frac{ne}{2}$. Then $\Sigma_{n,h}$ is very ample and 
 the image of the morphism $\phi_{\Sigma_{n,h}}\colon \mathbb 
 F_{e}\to \mP^{N}$ associated with $\Sigma_{n,h}$ is a smooth surface; hence
 by Bertini's theorem 
 the general member of $\Sigma_{n,h}$ is smooth. Then, by 
 adjunction, the genus of the general element 
 $\Gamma\in 
 \Sigma_{n,h}$ is 
 $$
 g_{n,h}=(n-1)(h-1)-\frac{ne}{2}(n-1).
 $$

 Moreover, by Riemann-Roch we have
 $$
 \dim \ \Sigma_{n,h}=g_{n,h}+2n+2h-1-ne.
 $$
 Following \cite[Theorem 12.16 see page 870]{ACG} we immediately 
 see that, given any fixed integer $n\geq 3$ and any fixed integer $e\geq 
 0$, the intervals 
 $$
 I_{n,h}=[g_{n,h}-n-h+1+\frac{ne}{2}, g_{n,h}]
 $$
 cover the half line $[0,\infty)$. Then there exist an integer 
 $\delta$ and an integer $h$ such that $0\leq\delta\leq 
 h+n-1-\frac{ne}{2}$ and
 \begin{equation}\label{enrico}
	g=g_{n,h}-\delta.
	\end{equation}
 A simple computation shows that
 $$
 \delta\leq g_{n,h}\leq {\rm{dim}}\Sigma_{n,h}-2\delta-1.
 $$
 By Castelnuovo's theorem applied to $\Sigma_{n,h}$,
 see for instance \cite[Theorem 12.6 pag.865]{ACG}, it follows that given 
 $\delta$ general points $a_{1},\ldots , a_{\delta}\in \mathbb 
 F_{e}$, there exists an irreducible curve $\Gamma\in 
 \Sigma_{n,h}=|nC_{0}+hl|$ 
 having $\delta$ nodes at $a_{1}$,\dots , $a_{\delta}$ and no other 
 singularities. Let $\nu\colon Z\to\mathbb F_{e}$ be the blow-up at 
 $a_{1}$,\ldots , $a_{\delta}$ and let $E_{i}:=\nu^{-1}(a_{i})$, 
 $i=1,.. ,\delta$. The normalisation $C$ of $\Gamma$ is contained 
 in $Z$ and $\nu_{|C}\colon C\to \Gamma$ is the normalisation 
 morphism. Let $H_{0}:=\nu^{-1}(C_{0})$ and 
 $\widetilde L:=\pi^{-1}(L)$. The smooth 
 curve $C$ has a $g^{1}_{n}$ induced by the ruling of $\mathbb 
 F_{e}$. Let us denote by $D$ a divisor of the $g^{1}_{n}$. Then
 $$
 D\sim L_{| C},\qquad C\in 
 |nH_{0}+h\widetilde L-\sum_{i=1}^{\delta}E_{i}|.
 $$
 By standard surface
 theory we get
 $$
 K_{Z}\sim -2H_{0}-(2+e)L+\sum_{i=1}^{\delta}E_{i}.
 $$ 
 Then $(n-2)H_{0}+(h-e-2-v)L-\sum_{i=1}^{\delta}E_{i}\sim 
 K_{Z}+C-vL$ and by adjunction theory on surfaces we have the following 
 exact sequence
 \begin{equation}\label{arba}
 0\to\sO_{Z}(K_{Z}-vL)\to
 \sO_{Z}(K_{Z}+C-vL)
 \to\omega_{C}(-vD)\to 0
\end{equation}
Notice now that by projection formula and by Serre duality we have
$$
h^{1}(Z,\sO_{Z}(K_{Z}-vL))=h^{1}(Z,\sO_{Z}(vL))=
h^{1}(\mP^{1},\sO_{\mP^{1}}(vP))=h^{0}(\mP^{1},\sO_{\mP^{1}}((-2-v)P))=0,
$$
since $v\geq 0$. As $h^{0}(Z,\sO_{Z}(K_{Z}-vL))=0$, by considering the 
long cohomology sequence associated with the sequence \ref{arba}, we obtain that the restriction 
morphism 
$$
H^{0}(Z,\sO_{Z}(K_{Z}+C-vL))\to H^{0}(C,\omega_{C}(-vD))
$$
is an isomorphism. In particular
$h^{0}(C,\omega_{C}(-vD))=h^{0}(Z,\sO_{Z}(K_{Z}+C-vL))$.

Since $a_{1},\ldots , a_{\delta}$ are general points,
we have
 $$
 {\rm{max}}\{-1, 
 h^{0}(Z,\sO_{Z}(K_{Z}+C-vL)\}=
 {\rm{max}}\{-1,{\rm{dim}}\ \Sigma_{n-2,h-(e+2+v)} -\delta\}.
 $$
 Hence if we assume that $h-(e+2+v)\geq \frac{(n-2)e}{2}$, then 
 ${\rm{dim}}\ \Sigma_{n-2,h-(e+2+v)} -\delta=g-(n-1)v$.

 We have 
 shown that if $h-(e+2+v)\geq \frac{(n-2)e}{2}$ and if $k$ is the unique 
integer such that $1\leq k\leq n-2$ and $g=m(n-1)+k$ then 
 \begin{equation}\label{beniaminono}
 h^{0}(C,\omega_{C}(-vD))= (m-v)(n-1)+k
 \end{equation}
if $m\geq v$.

 Finally, let 
 $p\colon C\to\mP^{1}$ be the gonal morphism;
then $p_{\star}\omega_{C}=\omega_{\mP^{1}}\oplus \sE_{C}(-2)$, where 
$$
\sE_{C}=\sO_{\mP^{1}}(m_{1})\oplus\sO_{\mP^{1}}(m_{2})\oplus\ldots \oplus\sO_{\mP^{1}}(m_{n-1}),
$$
with 
$m_{1}\leq m_{2}\leq\ldots \leq m_{n-1}$. By projection formula and 
by the equation \ref{beniaminono} it 
follows that in fact $m_{1}=\ldots =m_{n-1-k}=m-1$ and 
$m_{n-k}=\ldots=m_{n-1}=m$, that is $\sE_{C}$ is balanced.
\end{proof}

\section{Sharpness results}

In this Section we construct some examples, that realize the bound on the slope given in Proposition \ref{ssunif}.

\subsection{Existence of fibrations with the required properties: rational basis}
\label{subonesharp}

Let $Y=\mP^1\times\mP^1$ and $T=Y\times\mP^1$. Let $\pi_i\colon T\to \mP^1$ be the projection with respect to the $i$-th factor and set $\sL_i:=\pi_{i}^{\star}\sO_{\mP^1}(1)$ where $i=1,2,3$. Let 
$$
S\in |nL_1+nL_2+nL_3|
$$ 
 be a general element, where $n\geq 3$. If $\langle x_0,x_1\rangle=H^0(T,\sL_{3})$,
 $\langle y_0,y_1\rangle=H^0(T,\sL_{2})$ and $\langle z_0,z_1\rangle=H^0(T,\sL_{1})$, then 
 $$S=V(F)\,\ {\rm{where}}\,\, F=\sum_{i=0}^{n}a_i((y_0:y_1),(z_0:z_1))x_0^{n-i}x_1^{i}.$$
 Thus if $a_i((y_0:y_1),(z_0:z_1))\in H^0(Y,\sO_Y(n,n))$ for $i=0,\dots,n$ are general, the morphism $\pi\colon S\to Y$ is finite of degree $n$. Consider now the composition $f\colon S\to\mP^1$ of the inclusion $j\colon S\hookrightarrow T$ with
 the natural morphism $\rho\colon T\to Y$ followed by the projection on the first factor $\pi'_1\colon Y\to \mP^1$. The fiber over $z_0=a,z_1=b$ is the curve 
 $C_{[a:b]}=V(F_{[a:b]})\,\ {\rm{where}}\,\, F_{[a:b]}=\sum_{i=0}^{n}a_i((y_0:y_1),(a:b))x_0^{n-i}x_1^{i}$ inside $\mP^1\times\mP^1$. In particular $f\colon S\to\mP^1$ is a fibration in curves of genus $g=(n-1)^2$ and gonality $n$ such that 
$$
s(f)=\sF(n,g).
$$
Moreover $\Lambda_{\pi}$ is effective and it is induced on $S$ by the linear system $(0,0,n-2)$, hence $\Lambda_{\pi}^2=0$. Finally by performing the push-forward of the standard exact sequence
$$
0\to\sO_T(K_{T|Y})\to\sO_{T}(K_{T|Y}+S)\to \omega_{S|Y}\to 0,
$$
by projection formula and by relative duality it follows that the \tsc\ sheaf 
satisfies $\sE\cong  \sO_{Y}(n,n)^{\oplus n-1}$, so it is a uniform and balanced vector bundle which is also semistable on the $(0,1)$-sections of the projection $\pi'_1\colon Y\to\mP^1$. This shows that the bound given in Proposition \ref{ssunif} is sharp. 
Note that instead of $S\in |nL_1+nL_2+nL_3|$ we can take $S\in |n_1L_1+n_2L_2+n_3L_3|$ where $n_i\geq 1$ to obtain similar results.

\subsection{Existence of fibrations with the required properties: other cases}
\label{subonesharp_other}
%Let $Y=C_1\times C_2$ and $\pi_i\colon S\to C_i$, $i=1,2$. Set $\sL:= \pi^{\star}\sO_{C_1}(L_1)\otimes \pi^{\star} \sO_{C_2} (L_2)$ where $L_i$ is a very ample divisor on $C_i$ of degree $l_i$, $i=1,2$. Let $F\in |\sL|$. Then $F^2=2l_1l_2$. By the same method we know that the general pencil of $|\sL|$ is a semistable one with $2l_1l_2$ simple base points. Let $\sigma\colon S\to Y$ be the blow-up of these $2l_1l_2$ simple base points. It remains defined $f\colon S\to\mP^1$ a semistabe fibration such that:
%$$
%s(f)=8-\frac{2l_1l_2 }{(g_1-1)(g_2-1)+l_1l_2+(g_1-1)l_2+(g_2-1)l_1}
%$$
%Now assuming $g_2=0$ we have a fibration, over a $\mP^1$ different from $C_2$ of curves of gonality $l_1$.{\bf{Vale va verso 6 per $l_1$ che cresce!!}}. 
%
%
%This case too can be refined by considering 

Let $T=C_1\times\mP^1\times\mP^1$, where $C_1$ is a smooth curve of genus $g_1 >0$, and let
$\sL:=\pi^{\star}\sO_{C_1}(L_1)\otimes \pi_{2}^{\star}\sO_{\mP^1}(n)\otimes \pi_{3}^{\star}\sO_{\mP^1}(n)$. As above we obtain a semistable $n$-gonal fibration over $C_1$ of genus $g=(n-1)^2$ such that:
$$
s(f)=6-\frac{6}{(n-1)^2}=\sF(n,g)+4\frac{n-2}{(n-1)^2}
$$

Also in this case $\Lambda_\pi^2$ is zero and the Tschirnhausen sheaf is uniform and with balanced fiber restriction. 

%Consider now $A$ an abelian surface and set $T=A\times \mP^1$ Let $\sL$ and $\sM$ two very ample line bundles on $A$ and take $Y\in |\pi_{1}^{\star}\sL+\pi_{2}^{\star}(nP)|$, $n\in\mathbb Z, P\in\mP^1$. Let $D\in |\pi_{1}^{\star}\sM+\pi_{2}^{\star}(mP)|$. Define $F:=D_{|Y}$. Since $F^2=D^2Y=2m\sL\cdot\sM+n\sM^2$, if $\sigma\colon S\to Y$ is the blow-up of the simple $F^2$ intersection points of a general pencil on $Y$ having $F$ as its element we obtain a semistable fibration $f\colon S\to\mP^1$ such that
%
%$$
%s(f)=2\frac{(4m+3n-4)\sL^2+(8(n+m)-10)\sL\sM+3n\sM^2}{ (n-1+m)\sL^2+2(n+m)-2)\sL\sM+n\sM^2}
%$$
%If we take $\sL=\sM$ then
%$$
%s(f)=2\frac{14n+12m-10}{5n+3m-2}
%$$
%and we see that for $m$ going to infinity the limit is $8$ while for $n$ going to infinity the limit is $6-\frac{2}{5}$. If $n=m$ then it goes to $\frac{13}{2}${\bf{Vale: Sarebbero da fare i confronti per fare il confronto con la nostra situazione}}.
%

\end{document}